\documentclass[12pt]{amsart}
\usepackage{amssymb}
\usepackage{amsmath}
\usepackage{mathtools}
\usepackage[all]{xy}
\usepackage{enumitem}
\usepackage{xcolor}

 \rm

\renewcommand{\(}{\left(}
\renewcommand{\)}{\right)}
\renewcommand{\[}{\left[}
\renewcommand{\]}{\right]}

\newcommand{\abs}[1]{\left\lvert#1\right\rvert}
\newcommand{\norm}[1]{\left\lVert#1\right\rVert}
\newcommand{\st}{\:|\:}

\newcommand{\C}{{\mathbb{C}}}
\newcommand{\R}{{\mathbb{R}}}

\newcommand{\e}{\varepsilon}
\renewcommand{\phi}{\varphi}
\newcommand{\p}{\rho}

\renewcommand{\H}{{\mathcal{H}}}

\newcommand{\BH}{{\mathcal{B}}(\H)}

\theoremstyle{plain}
\newtheorem{thm}{Theorem}[section]
\newtheorem{lem}[thm]{Lemma}

\newtheorem{assumption}[thm]{Assumption}

\theoremstyle{definition}
\newtheorem{defn}[thm]{Definition}

\theoremstyle{remark}

\title{The Weyl Transform of a measure}

\author{Mansi Mishra and M.~K.~Vemuri}
\address{Department of Mathematical Sciences, IIT(BHU), Varanasi 221005.}

\date{\today}

\begin{document}

\begin{abstract}
(1) Suppose $\mu$ is a smooth measure on a hypersurface of positive Gaussian
curvature in $\R^{2n}$.  If $n\ge 2$, then $W(\mu)$, the Weyl transform of
$\mu$, is a compact operator, and if $p>n\ge 6$ then $W(\mu)$ belongs
to the $p$-Schatten class.

(2) There exist Schatten class operators with
linearly dependent quantum translates.
\end{abstract}

\keywords{Curvature; Heisenberg group; Schatten class.}
\subjclass[2010]{22D10, 22E30, 43A05, 43A80, 47B10}

\maketitle
\thispagestyle{empty}

\section{Introduction}
Suppose $S$ is a smooth hypersurface in $\R^n$, $n \geq 2$,
whose Gaussian curvature is nonzero everywhere.
By a {\em smooth measure on $S$} we mean a measure of the form
$\mu=\psi\sigma$ where $\sigma$ is the measure on $S$ induced by
the Lebesgue measure on $\R^n$ and $\psi$ is a smooth function
on $\R^n$ whose support intersects $S$ in a compact set.
In Stein's book \cite[p348, Theorem 1]{MR1232192}
(see also \cite[Theorem 7.7.14]{MR1065993}),
it is proved that for such a measure,
\begin{equation*}
\abs{\widehat{\mu}(\xi)} \leq A \norm{\xi}^{(1-n)/2},
\end{equation*}
where $\widehat{\mu}$ denotes the Fourier transform of the measure $\mu$.
A brief history of this problem is given in Stein's book.
It follows from this inequality that, for $p > 2n/{(n-1)}$,
\begin{equation}\label{E:Lp}
\widehat{\mu} \in L^p(\R^n).
\end{equation}


Let $\H = L^2\(\R^n\)$, and $\BH$ the set of bounded operators on $\H$.
If $f \in L^1(\R^{2n})$, the {\em Weyl Transform} of $f$ is the operator 
$W(f) \in \BH$ defined by 
\begin{equation*}
(W(f)\phi)(t)
=
\iint f(x,y) e^{\pi i(x \cdot y+2y \cdot t)} \phi(t+x) \, dx \, dy.  
\end{equation*}
More generally, if $\lambda$ is a finite measure on $\R^{2n}$,
the {\em Weyl Transform} of $\lambda$ is the operator $W(\lambda) \in \BH$ 
defined by 
\begin{equation*}
(W(\lambda)\phi)(t)
=
\int e^{\pi i(x \cdot y+2y \cdot t)} \phi(t+x) \, d\lambda(x,y).  
\end{equation*}

The Weyl transform is closely related to the group Fourier transform
on the Heisenberg group.
Recall that the Heisenberg group $G$ is the set of triples
\begin{equation*}
\{(x,y,z) | x,y \in \R^n, z \in \C, \abs{z} = 1\}
\end{equation*}
with multiplication defined by
\begin{equation*}
(x,y,z)(x',y',z')=\(x+x', y+y', zz'e^{\pi i (x \cdot y'-y \cdot x')}\).
\end{equation*}
According to the {\em Stone-von Neumann Theorem}, there is a unique irreducible
unitary representation $\rho$ of $G$ such that
\begin{equation*}
\rho(0,0,z)= zI.
\end{equation*}
The standard realization of this representation is on the Hilbert space $\H$ by
the action
\begin{equation}\label{E:schrodinger}
\(\rho(x,y,z)\phi\)(t)= ze^{\pi i (x\cdot y+2y\cdot t)}\phi(t+x).
\end{equation}
Thus, the Weyl transform of $\lambda$ may be expressed as
\begin{equation}\label{E:weyl}
W(\lambda)
=
\int \rho(x,y,1) \, d\lambda(x,y),  
\end{equation}
where the integral is the weak integral as defined in
\cite[Definition 3.26]{MR1157815}.
Let $S^{p}(\H)$ denote the $p^{\mathrm{th}}$ Schatten class of $\H$.
There are analogues of the Plancherel theorem and the Hausdorff-Young
theorem for the Weyl transform.  Namely, $W$ extends to an isometric
isomorphism $L^2(\R^{2n})\to S^2(\H)$, and to a bounded linear map
$L^p(\R^{2n}) \to S^{p'}$ where $1 \le p \le 2$ and $1/p+1/p'=1$
(see e.g. \cite[Theorem 1.30]{MR983366} and \cite[Theorem 1.3.3]{MR1633042}).


The main result of this paper is the following theorem, which is an analogue of
Equation (\ref{E:Lp}) for the Weyl transform, and is proved in
Section \ref{S:proof}.
\begin{thm}\label{T:main-thm}
Suppose $S$ is a compact connected hypersurface in $\R^{2n}$, $n \geq 2$,
whose Gaussian curvature is positive everywhere.  
Let $\mu$ be a smooth measure on $S$.
Then $W(\mu)$ is a compact operator.  Moreover, when $n\geq 6$ and $p > n$, 
\begin{equation*}
W(\mu) \in S^{p}(\H).
\end{equation*}
\end{thm}

We conjecture that $W(\mu) \in S^p(\H)$ if $p > 4n/{(2n-1)}$, with no
restriction on $n$.
We cannot do better because of the following example.
Let $\mu_r$ denote the normalized surface measure on the sphere
$\abs{z}=r$ in $\C^n$.  In \cite[Proposition 4.1]{MR1133375}, Thangavelu
proved that
\begin{equation*}
W(\mu_r)=\sum_{k=0}^{\infty} \frac{k!(n-1)!}{(k+n-1)!}\phi_k (r) P_k ,
\end{equation*}
where $\phi_k$ are the Laguerre functions of type $(n-1)$ and
$P_k$ is the orthogonal projection of $\H$ onto the $k^{\mathrm{th}}$ eigenspace
of the harmonic oscillator.  It follows from the asymptotic properties of
the Laguerre functions (see \cite[Theorem 8.22.1]{MR0000077}) that,
\begin{equation*}
W(\mu_r) \in S^{p}(\H)
\end{equation*}
if and only if $p > 4n/{(2n-1)}$.

In \cite[Example 2.4]{MR536947}, Edgar and Rosenblatt proved that there exists a
function in $L^p(\R^n)$, $n \geq 2$, that has linearly dependent translates
when $p > 2n/(n-1)$.
In Section \ref{S:application},
Theorem \ref{T:main-thm} is applied to prove an analogue of the result in
\cite{MR536947} to quantum translates of a Schatten class operator.

\section{The proof}\label{S:proof}
Let $S$ be a compact connected hypersurface in $\R^{2n}$.
Assume that $S$ has a positive Gaussian curvature everywhere.
Observe that this actually implies that $S$ has positive definite
second fundamental form, i.e., $S$ is strictly convex
(see \cite[p122, Theorem 2.1]{MR707850}).

Let $\Pi_i : S \to \R^n$, $i=1,2$, be the projection maps on $S$ defined by 
\begin{equation*}
\begin{aligned}
\Pi_1(x_1, \dots, x_n, y_1, \dots, y_n)
=&\; (x_1, \dots, x_n), \quad \text{and}\\ 
\Pi_2(x_1, \dots, x_n, y_1, \dots, y_n)
=&\; (y_1, \dots, y_n).  
\end{aligned}
\end{equation*}

We will first prove Theorem \ref{T:main-thm} under the following
additional assumption.

\begin{assumption}\label{A:assumption}
The support of $\mu$ is contained in the set of regular points of
$\Pi_1$.
\end{assumption}

Let $U_1$ denote the set of regular points of $\Pi_1$ and
$$
W_1=\{x \in \Pi_1(S) \st \text{$x$ is a regular value for $\Pi_1$}\}.
$$
Then $W_1$ is open by the implicit function theorem.
Observe that $\Pi_1(U_1)=W_1$.  Indeed, if $x$ is the image under
$\Pi_1$ of a regular point of $\Pi_1$, then $x$ cannot be the image
of a critical point of $\Pi_1$ because $S$ is strictly convex.

Let $S_x=\Pi_1^{-1}\{x\}$.  We need the following lemma, whose proof
is given in the appendix.

\begin{lem}\label{L:sff}
Whenever $x\in W_1$, the set $\Pi_2(S_x)$ is a
smooth hypersurface in $\R^n$ whose Gaussian curvature is bounded below
by a positive constant independent of $x$.
\end{lem}

Let $\mathcal{J}_{\Pi_1}$ denote the normal Jacobian of $\Pi_i$, i.e.,
the absolute value of the determinant of $d\Pi_1$ restricted to the
orthogonal complement of its kernel.  Then $\mathcal{J}_{\Pi_1}$ is
strictly positive on $U_1$ and hence
has a positive minimum on the support of $\mu$.

By the coarea formula (see \cite[p159]{MR2229062}), we have
\begin{equation*}
\begin{split}
(W(\mu)\phi)(t)
& = 
\int_{S} e^{\pi i(x \cdot y+2y \cdot t)} \phi(t+x) \,d\mu(x,y)\\
& = 
\int_{S} e^{\pi i(x \cdot y+2y \cdot t)}\phi(t+x)\psi(x,y)\,d\sigma(x,y)\\
& = 
\int_{\R^n} \phi(t+x) \int_{S_x}
\frac{e^{\pi i(x +2t) \cdot y}\psi(x,y)}{\mathcal{J}_{\Pi_1}(x,y)}
\,d\sigma_x(y)\,dx,
\end{split}
\end{equation*}
where $\sigma_x$ denotes the volume measure on $S_x$.
Let $\psi_1 = \psi / \mathcal{J}_{\Pi_1}$.  Then $\psi_1$ is a compactly
supported smooth function on $U_1$.
Observe that
\begin{equation*}
\int_{S_x} \frac{e^{\pi i(x +2t) \cdot y}\psi(x,y)}{\mathcal{J}_{\Pi_1}(x,y)}
\,d\sigma_x(y) =
\widehat{\eta_x}(x+2t),
\end{equation*}
where $\eta_x = {\Pi_2}_* \(\psi_1(x,\cdot) \, d\sigma_x\)$ is the pushforward
of the measure $\psi_1(x,\cdot) \, d\sigma_x$ by $\Pi_2$, which is a smooth
measure supported on $\Pi_2(S_x)$.  Therefore, by Lemma \ref{L:sff} and
\cite[Theorem 7.7.14]{MR1065993}, there exists a constant $C$ independent of 
$x$ such that
\begin{equation*}
\abs{\widehat{\eta_x}(t)}\leq C {\norm{t}}^{(1-n)/2}.
\end{equation*}
Also, since the measure $\eta_x$ is nonzero only when $x\in \Pi_1(S)$, it
follows that $\widehat{\eta_x}(t)$ is nonzero only when $x\in \Pi_1(S)$.

Thus $W(\mu)$ is an integral operator given by the kernel
\begin{equation*}
k(t,u) = \widehat{\eta_{u-t}}(u+t),
\end{equation*}
and the kernel $k$ satisfies
\begin{equation}\label{eq:kernel}
\abs{k(t,u)}\leq C {\norm{t+u}}^{(1-n)/2}
\end{equation}
and $k(t,u)=0$ if $\norm{t-u}>R$ for some constants $C$ and $R$.  We
will now show that such an integral operator is compact.

\begin{lem}\label{L:conv}
Let $R >0$ and $\alpha <0$.  If $f \in L^1_{loc} (\R^n)$,
$f(t)=O({\norm{t}}^\alpha)$ as $\norm{t} \rightarrow \infty$ and $g=f*\chi_{B_R}$
then $g(t)=O({\norm{t}}^\alpha)$ as $\norm{t} \rightarrow \infty$.
\end{lem}
\begin{proof}
Since $f(t)=O({\norm{t}}^\alpha)$ as $\norm{t} \rightarrow \infty$,
there exists a constant $C$ such that
\begin{equation*}
\abs{f(t)} \leq C \norm{t}^\alpha
\end{equation*}
for large enough $t$.  Then 
\begin{equation*}
\begin{split}
\abs{g(t)} & = \abs{\int f(t-u)\chi_{B_R}(u) \,du}\\
&  
\leq \int_{B_R} \abs{f(t-u)} \,du\\
& 
\leq C \int_{B_R} \norm{t-u}^\alpha \,du.\\
\end{split}
\end{equation*}
It follows from the mean value theorem that there exists
$\lambda \in [0,1]$ such that
\begin{equation*}
\abs{\norm{t-u}^\alpha - \norm{t}^\alpha}
\leq
\alpha \norm{t-\lambda u}^{\alpha -1} \norm{u}.
\end{equation*}
Therefore 
\begin{equation*}
\norm{t-u}^\alpha \leq
\norm{t}^\alpha + \alpha R \(\norm{t}-\lambda R\)^{\alpha -1}.
\end{equation*}
Hence, for large enough $t$, there exists a constant $C'$ such that 
\begin{equation*}
\abs{g(t)} \leq C' \norm{t}^\alpha.
\end{equation*}
\end{proof}
Observe that
\begin{equation*}
\begin{split}
\int \abs{k(t,-u)}\, dt & = \int_{B_r(-u)} \abs{k(t,-u)}\, dt\\
& 
\leq C \int_{B_r(-u)} \norm{u-t}^{(1-n)/2} \,dt\\
&
\leq C \int \chi_{B_r(-u)}(t) \norm{u-t}^{(1-n)/2} \,dt\\
&
= C \norm{\cdot}^{(1-n)/2} * \chi_{B_r(-u)}.
\end{split}
\end{equation*}
Therefore by Lemma \ref{L:conv}, for large enough $u$,
\begin{equation*}
\int\abs{k(t,-u)} \, dt \leq C' \norm{u}^{(1-n)/2}.
\end{equation*}
Hence
\begin{equation*}
\lim_{\norm{u} \rightarrow \infty} \int \abs{k(t,u)}\,dt = 0.
\end{equation*}
Similarly, we can conclude that 
\begin{equation*}
\lim_{\norm{t} \rightarrow \infty} \int \abs{k(t,u)}\,du = 0.
\end{equation*}
Compactness of $W(\mu)$ follows from the generalized Young's inequality
(see \cite[p9]{MR1357411}).

Recall that the Laplacian is the differential operator
\begin{equation*}
\Delta =\sum_{i=0}^n\frac{\partial^2}{\partial t^2_i} .
\end{equation*}
\begin{defn}
The harmonic oscillator is the differential operator  
\begin{equation*}
H = -\Delta + \norm{t}^2.
\end{equation*}
\end{defn}
The Hermite functions $\Phi_\alpha$,
for $\alpha = (\alpha_1, \dots, \alpha_n)$ a multi-index,
are the eigenfunctions of the harmonic oscillator corresponding to the
eigenvalues $(2\abs{\alpha} +n)$, where $\abs{\alpha} = \sum_{j=1}^n \alpha_j$
(see \cite[p18]{MR1633042}).
Therefore, for $p > n$, 
\begin{equation}
H^{-1} \in S^p(\H).
\label{eq:ho}
\end{equation}
It follows from the dominated convergence theorem that
\begin{equation*}
H(W(\mu)\phi)(t) = \int (-\Delta_t + \norm{t}^2) k(t,u) \phi(u)\, du.
\end{equation*}

Let $i_x : S_x \to U_1$ denote the inclusion map.  Observe that there exists
a differential form $\omega$ on $U_1$ such that the measure on $S_x$
induced by $i_x^* \omega$ is $\sigma_x$, where
$i_x^*\omega$ denotes the pullback of $\omega$ by $i_x$.


Fix a point $x \in W_1$.  Let $\tau \in C_0^\infty(W_1)$ be identically one
in a neighborhood of $x$.  There exists a unique vector field $X_j$ on $U_1$
such that $X_j \perp \mathrm{ker}(d\Pi_1)$ and 
$d\Pi_1(X_j) = \tau \frac{\partial}{\partial x_j}$.  Clearly
the support of $X_j$ is a compact subset of $U_1$, and so there is a
one parameter group $\phi_s$ of diffeomorphisms of $U_1$ generated by
$X_j$.
Then, at the point $x$, we have
\begin{equation*}
\begin{split}
\frac{\partial}{\partial x_j} \int_{S_x} e^{\pi i (x+2t)\cdot y} \psi_1(x,y)\,d\sigma_x
&=
\frac{\partial}{\partial x_j} \int_{S_x} e^{\pi i (x+2t)\cdot y} \psi_1(x,y)\, \omega\\
&=
\int_{S_x}\frac{d}{ds}\bigg|_{s=0} \phi_s^*
(e^{\pi i (x+2t)\cdot y} \psi_1(x,y)\, \omega)\\
&=
\int_{S_x} \mathcal{L}_{X_j} \(e^{\pi i (x+2t)\cdot y} \psi_1(x,y)\, \omega\),
\end{split}
\end{equation*}
where $\mathcal{L}_{X_j}(\nu)$ denotes the Lie derivative of the differential
form $\nu$ with respect to the vector field $X_j$.
We may write
$
X_j = \tau(x) \frac{\partial}{\partial x_j} +
a_{1j}(x,y) \frac{\partial}{\partial y_1} + \cdots +
a_{nj}(x,y) \frac{\partial}{\partial y_n}
$,
where $a_{1j}, \dots, a_{nj} \in C_0^{\infty}(U_1)$.
Since $i_x^* \omega$ is a non-vanishing $(n-1)$-form on the $(n-1)$-dimensional
manifold $S_x$, we may write
$i_x^* \mathcal{L}_{X_j}\omega = f(x,y) i_x^* \omega$ for some
$f \in C^{\infty}(U_1)$.  Since $X_j$ is compactly supported, in fact
$f\in C_0^{\infty}(U_1)$.  Therefore
\begin{equation*}
\begin{aligned}
 &\; i_x^*\mathcal{L}_{X_j} \[e^{\pi i (x+2t)\cdot y} \psi_1(x,y)\, \omega\]\\
=&\;
     \(X_je^{\pi i (x+2t)\cdot y}\) \psi_1(x,y) i_x^*\omega
     + e^{\pi i (x+2t)\cdot y} \(X_j\psi_1\)(x,y) i_x^*\omega
     + e^{\pi i (x+2t)\cdot y} \psi_1(x,y)i_x^*\mathcal{L}_{X_j}\omega\\
=&\;
     \pi i \(\tau(x) y_j + \sum_{k=1}^n(x_k + 2t_k) a_{kj}(x,y)\)
     e^{\pi i (x+2t)\cdot y}\psi_1(x,y)\,d\sigma_x
     + e^{\pi i (x+2t)\cdot y}\(X_j\psi_1\)(x,y)\,d\sigma_x\\
&\;  + e^{\pi i (x+2t)\cdot y}\psi_1(x,y)f(x,y)\,d\sigma_x\\ 
=&\;
     \pi i \sum_{k=1}^n(x_k + 2t_k) a_{kj}(x,y)
     e^{\pi i (x+2t)\cdot y}\psi_1(x,y)\,d\sigma_x\\
&\;  + e^{\pi i (x+2t)\cdot y}
     \[\pi i\tau(x) y_j\psi_1(x,y)+\(X_j\psi_1\)(x,y)+\psi_1(x,y)f(x,y)\]
     \,d\sigma_x.
\end{aligned}
\end{equation*}
By Lemma \ref{L:sff} and \cite[Theorem 7.7.14]{MR1065993},
there exists constants $C_1$ and $C_2$
such that
\begin{equation*}
\begin{aligned}  
&\;
\abs{\int_{S_x} a_{kj}(x,y)e^{\pi i (x+2t)\cdot y}\psi_1(x,y)\,d\sigma_x}
\leq
C_1 {\norm{x+2t}}^{(1-n)/2}, \quad \text{and}\\
&\;
\abs{\int_{S_x}  e^{\pi i (x+2t)\cdot y}
\[\pi i\tau(x) y_j\psi_1(x,y)+\(X_j\psi_1\)(x,y)+\psi_1(x,y)f(x,y)\]\,d\sigma_x}
\leq
C_2 {\norm{x+2t}}^{(1-n)/2},
\end{aligned}
\end{equation*}
for large enough $t$.  Therefore
\begin{equation*}
\abs{\frac{\partial}{\partial x_j} \int_{S_x}
e^{\pi i (x+2t)\cdot y} \psi_1(x,y)\,d\sigma_x}
\leq
C_3 {\norm{x+2t}}^{(3-n)/2},
\end{equation*}  
for large enough $t$.  Similarly,
\begin{equation*}
\abs{\frac{\partial^2}{\partial {x_j}^2} \int_{S_x}
e^{\pi i (x+2t)\cdot y} \psi_1(x,y)\,d\sigma_x}
\leq
C_4 {\norm{x+2t}}^{(5-n)/2},
\end{equation*}
for large enough $t$.  Therefore
\begin{equation*}
\abs{\Delta_x \int_{S_x}
e^{\pi i (x+2t)\cdot y} \psi_1(x,y)\,d\sigma_x}
\leq
C_5 {\norm{x+2t}}^{(5-n)/2},
\end{equation*}
for large enough $t$.  It follows that
\begin{equation*}
\abs{\Delta_t k(t,u)}\leq C_5 {\norm{t+u}}^{(5-n)/2},
\end{equation*}
for large enough $t$.  Therefore, by Equation (\ref{eq:kernel}) and
Lemma \ref{L:conv},
\begin{equation*}
\begin{aligned}
\int \abs{(-\Delta_t + \norm{t}^2)k(t,u)}\,du 
&\; \leq C_5 \int{\norm{t+u}}^{(5-n)/2}\,du
    + C \norm{t}^2 \int {\norm{t+u}}^{(1-n)/2}\,du\\
&\; \leq C_6 {\norm{t}}^{(5-n)/2} 
\end{aligned}
\end{equation*}
for large enough $t$.  Hence, when $n \geq 6$, we have
\begin{equation*}
\lim_{\norm{t} \rightarrow \infty} \int \abs{(-\Delta_t + \norm{t}^2)k(t,u)}\,du  = 0.
\end{equation*}
Similarly, when $n \geq 6$, we have
\begin{equation*}
\lim_{\norm{u} \rightarrow \infty} \int \abs{(-\Delta_t + \norm{t}^2)k(t,u)}\,dt = 0.
\end{equation*}
Therefore $H(W(\mu))$ is a compact operator when $n \geq 6$.
Since $S^p(\H)$ is an ideal of $\BH$, it follows from Equation (\ref{eq:ho})
that 
\begin{equation*}
W(\mu) \in S^p(\H)
\end{equation*}
for $n \geq 6$ and $p > n$.

This completes the proof of Theorem \ref{T:main-thm} under Assumption
\ref{A:assumption}.

Now we consider the general case.

\begin{lem}\label{L:sub}
Let $w \in S$.  Then either $\Pi_1$ or $\Pi_2$ is a submersion at $w$.
\end{lem}
\begin{proof}
Let us assume that $d\Pi_1$ and $d\Pi_2$ are not surjective i.e.,
\begin{equation*}
\operatorname{Rank}(d\Pi_1) \leq n-1
\quad \text{and} \quad
\operatorname{Rank}(d\Pi_2) \leq n-1.
\end{equation*}
Therefore, by the rank-nullity theorem,
\begin{equation*}
\operatorname{Nullity}(d\Pi_1) \geq n
\quad \text{and} \quad
\operatorname{Nullity}(d\Pi_2) \geq n.
\end{equation*}
It follows that the intersection of the null-spaces $d\Pi_1$ and
$d\Pi_2$ is nonzero.  This contradicts the fact that
$d\Pi_1 \oplus d\Pi_2$ is injective.
\end{proof}

Let $U_2$ be the set of regular points of $\Pi_2$.
By Lemma \ref{L:sub}, $S = U_1 \cup U_2$.
Let $\{\rho_1, \rho_2\}$ be a partition of unity subordinate to $\{U_1, U_2\}$.

Then the measure $\p_1\mu$ satisfies Assumption \ref{A:assumption}, and
hence $W(\p_1\mu)$ satisfies the conclusions of Theorem \ref{T:main-thm}.

Define $J: \R^{2n} \to \R^{2n}$ by
\begin{equation*}
J(x,y) = (-y,x).
\end{equation*}

Then $JS$ is also a compact connected hypersurface in $\R^{2n}$, $n \geq 2$,
whose Gaussian curvature is positive everywhere, and
$J_*(\p_2\mu)$ is a smooth measure on $JS$, which in addition satisfies
Assumption \ref{A:assumption}.  Therefore $W(J_*(\p_2\mu))$ 
satisfies the conclusions of Theorem \ref{T:main-thm}.

Let $\mathcal{F}$ denote the Fourier transform on $\H$.  It follows
from Equation (\ref{E:schrodinger}) that
\begin{equation*}
\mathcal{F} \rho(x,y,1) = \rho(-y,x,1) \mathcal{F}.
\end{equation*}
Therefore, by Equation (\ref{E:weyl}),
\begin{equation*}
W(\rho_2 \mu) = \mathcal{F}^{-1} W(J_*(\rho_2 \mu)) \mathcal{F},
\end{equation*}
and so $W(\rho_2 \mu)$ satisfies the conclusions of Theorem \ref{T:main-thm}.
Since $W(\mu)=W(\rho_1 \mu) + W(\rho_2 \mu)$, Theorem \ref{T:main-thm}
is proved.

\section{An application}\label{S:application}
There is an action of $\R^{2n}$ on $S^p(\H)$ called {\em quantum translation}
which is defined by
\begin{equation*}
(x_1,y_1) \cdot A = \rho (x_1,y_1,1) A {\rho (x_1,y_1,1)}^{-1} .
\end{equation*}
The motivation for this definition may be found in \cite[p28]{MR1228140}.
This action plays an important role in the works
\cite{MR1228140, MR2434197, MR3628015}.

It follows from Equation (\ref{E:weyl}) that for any finite measure
$\lambda$ on $\R^{2n}$,
\begin{equation}\label{E:weyl2}
(x_1, y_1)\cdot W(\lambda)=W(e((x_1,y_1),\cdot) \lambda),
\end{equation}
where $e: \R^{2n} \times \R^{2n} \to \C$ is the map
\begin{equation*}
e((x',y'),(x,y))=e^{2\pi i(x'\cdot y- y'\cdot x)}.
\end{equation*}
By the Plancherel theorem for $W$, this relation continues to hold when
$\lambda$ is replaced by a function $f\in L^2(\R^{2n})$.


\begin{lem}
Let $A \in S^p(\H)$, $1\le p \le 2$, be a non-zero operator and let
$(x_1,y_1), \dots, (x_k,y_k)$ be distinct elements of $\R^{2n}$.  Then
$\{(x_1,y_1) \cdot A, \dots, (x_k,y_k) \cdot A\}$ is a linearly independent set.
\end{lem}  

\begin{proof}
Since $S^p(\H) \subseteq S^2(\H)$ for $1\le p \le 2$, it suffices to prove
the result for $p=2$.

Let $c_1,\dots, c_k \in \C$ be such that
\begin{equation*}
c_1 (x_1,y_1) \cdot A + \dots + c_k (x_k,y_k) \cdot A = 0.
\end{equation*}

By the Plancherel theorem for $W$, we may write $A=W(f)$ for some nonzero
$f\in L^2(\R^{2n})$. 
By Equation (\ref{E:weyl2}), 
\begin{equation*}
c_1 (x_1,y_1) \cdot A + \dots + c_k (x_k,y_k) \cdot A
= W\((c_1 e((x_1,y_1),\cdot)+ \dots + c_k e((x_k,y_k),\cdot)) f\).
\end{equation*}
Therefore $(c_1 e((x_1,y_1),\cdot)+ \dots + c_k e((x_k,y_k),\cdot)) f=0$.
Since $f$ is a non-zero function in $L^2(\R^{2n})$, there exists $\e >0$
such that $\abs{f} > \e$ on a set of positive measure.
Therefore the analytic function
\begin{equation*}
c_1 e((x_1,y_1),(x,y))+ \dots + c_k e((x_k,y_k),(x,y)) = 0
\end{equation*}
on a set of positive measure and thus it is identically zero.
Since 
$\{e((x_1,y_1),\cdot), \dots, e((x_k,y_k),\cdot)\}$ is linearly independent,
it follows that $c_1 = \dots = c_k = 0$.
\end{proof}  

The remainder of the argument is taken from Edgar and
Rosenblatt \cite{MR536947}, which we give here for completeness.

Let $(x_1,y_1), \dots, (x_k,y_k)$ be distinct elements of $\R^{2n}$ and
let $c_1,\dots, c_k$ be nonzero scalars.
Define the difference operator $D$ on $S^p(\H)$ by
\begin{equation*}
DA = c_1 (x_1,y_1) \cdot A + \dots + c_k (x_k,y_k) \cdot A.
\end{equation*}
The trigonometric polynomial 
\begin{equation*}
c_1 e\((x_1,y_1),(x,y)\) + \dots + c_k e\((x_k,y_k),(x,y)\)
\end{equation*}
is called the characteristic trigonometric polynomial of the
difference equation $DA = 0$.

Consider the difference equation given by
\begin{equation}\label{eq:de}
2(2n-1) A = (1,0,\dots,0)\cdot A + (-1,0,\dots,0)\cdot A
+ \dots + (0,\dots, 0,1)\cdot A +(0,\dots, 0,-1)\cdot A.
\end{equation}
Then the characteristic trigonometric polynomial
\begin{equation*}
p(x_1, \dots, x_n, y_1, \dots, y_n) = 2(2n-1) - 2\sum_{j=1}^n \(\cos(2 \pi x_j) +
\cos(2\pi y_j)\)
\end{equation*}
has a zero set which is a disjoint union of compact $(2n-1)$-dimensional
surfaces of positive Gaussian curvature.  Let $S_{2n}$ be the connected component
of the zero set containing the points with all coordinates zero except for one
which is $\pm 1/4$.  Let $\sigma$ be the measure on $S_{2n}$
induced by the Lebesgue measure on $\R^{2n}$.
Then by Theorem \ref{T:main-thm}, $W(\sigma) \in S^p(\H)$ for $p>n$ and
$n \geq6$.
It follows from Equation (\ref{E:weyl2}) that $A=W(\sigma)$ is a
non-zero solution of the difference equation (\ref{eq:de}).
Hence, for $p>n$ and $n \geq6$, there exists an operator in $S^p(\H)$ whose
quantum translates are linearly dependent.

\section*{Appendix}
We prove Lemma \ref{L:sff} here.
We need another technical lemma.

\begin{lem}\label{L:technical}
Let $M$ be a compact connected hypersurface in $\R^m$ with positive Gaussian
curvature.  Let $H$ be a plane in $\R^m$ and suppose $H$ intersects $M$
transversally.  Let $N=M\cap H$.  Then $N$ is a
hypersurface in $H$, and for all $y\in N$ and $X, Y \in T_yN$,
we have
\begin{equation*}
K_{(N \subseteq H)}(X,Y) \geq K_{(M \subseteq \R^m)}(X,Y),
\end{equation*}
where $K_{(N \subseteq H)}$ and $K_{(M \subseteq \R^m)}$ are the second
fundamental forms with respect to the inward pointing normal
on the hypersurfaces $N$ and $M$ respectively.
\end{lem}

\begin{proof}
By the Jordan-Brower separation theorem, $M=\partial\Omega$
where $\Omega$ is a bounded open set in $\R^m$.
Let $\vec{n}_M$ denote the unit normal to $M$ which points into $\Omega$.
Let $N=M\cap H$.  It is a standard fact from differential topology that $N$
is a smooth hypersurface in $H$. (see \cite[p30]{MR0348781}).
Let $\vec{n}_N$ denote the
unit normal vector to $N$ which points into $\Omega\cap H$.  Observe that
$\vec{n}_N\cdot\vec{n}_M>0$.

If $W$ is a submanifold of the Riemannian manifold $Z$, let
$\vec{K}_{(W \subseteq Z)}$ denote the vector-valued second fundamental form of
$W$ in $Z$, i.e.
\begin{equation*}
\vec{K}_{(W \subseteq Z)}(X,Y)=\operatorname{Proj}_{TZ \ominus TW} \nabla_X Y,
\end{equation*}
where $\nabla$ is the Levi-Civita connection on $Z$ and
$TZ \ominus TW$ denotes the orthogonal complement of the tangent bundle of
$W$ in the tangent bundle of $Z$.  Moreover, if $\vec n$ is a unit
normal vector field along $W$, we may define the second fundamental
form of $W$ with respect to $\vec n$ by
\begin{equation*}
K_{(W\subseteq Z)}=\vec{K}_{(W\subseteq Z)}\cdot \vec{n}.
\end{equation*}



Let $x\in N$ and let $X, Y \in T_xN$.  We have orthogonal decompositions
\begin{equation*}
\begin{aligned}
\vec{K}_{(N \subseteq \R^m)}(X,Y)
=&\; \vec{K}_{(N \subseteq H)}(X,Y) + \vec{K}_{(H \subseteq \R^m)}(X,Y),
\qquad\text{and}\\
\vec{K}_{(N \subseteq \R^m)}(X,Y)
=&\; \vec{K}_{(N \subseteq M)}(X,Y) + \vec{K}_{(M \subseteq \R^m)}(X,Y).
\end{aligned}
\end{equation*}
Since $\vec{K}_{(H \subseteq \R^m)}(X,Y)=0$, it follows that
\begin{equation*}
\vec{K}_{(N \subseteq H)}(X,Y)
=
\vec{K}_{(N \subseteq M)}(X,Y) + \vec{K}_{(M \subseteq \R^m)}(X,Y)
\end{equation*}

Since $\vec{K}_{(N \subseteq M)}(X,Y)$ is tangent to $M$, it follows that
\begin{equation*}
\begin{aligned}
K_{(N \subseteq H)}(X,Y) \vec{n}_N \cdot \vec{n}_M
=&\; \vec{K}_{(N \subseteq H)}(X,Y) \cdot \vec{n}_M\\
=&\; \vec{K}_{(M \subseteq \R^m)}(X,Y) \cdot \vec{n}_M\\
=&\; K_{(M \subseteq \R^m)}(X,Y).
\end{aligned}
\end{equation*}
Therefore
\begin{equation*}
K_{(N \subseteq H)}(X,Y)
=
\frac{K_{(M \subseteq \R^m)}(X,Y)}{ \vec{n}_M\cdot \vec{n}_N}
\geq
K_{(M \subseteq \R^m)}(X,Y).
\end{equation*}
\end{proof}

Taking $H$ to be the plane $\{(x,y) \st y \in \R^n\}$ in
Lemma \ref{L:technical}, we see that $S_x$ is a smooth
hypersurface in $H$ with 
$K_{(S_x \subseteq H)}(X,Y) \geq K_{(S \subseteq \R^{2n})}(X,Y)$.
Since $\Pi_2$ is an isometry between $H$ and $\R^n$,
it follows that $\Pi_2(S_x)$ is a hypersurface in $\R^n$
with
$K_{(\Pi_2(S_x) \subseteq \R^n)}(X,Y) \geq K_{(S \subseteq \R^{2n})}(X,Y)$.
Therefore the smallest principal curvature of $\Pi_2(S_x)$ is
greater than or equal to the smallest principal curvature of
$S$.  Therefore the Gaussian curvature of $S_x$ is bounded below
by the $(n-1)$-th power of the smallest principal curvature of
$S$.

\bibliographystyle{amsplain}
\bibliography{v21-weyltransform}

\end{document}